\documentclass[10pt]{amsart}

\usepackage{tikz-cd}

\usepackage[english]{babel}

\usepackage[letterpaper,top=2cm,bottom=2cm,left=3cm,right=3cm,marginparwidth=1.75cm]{geometry}

\usepackage{amsmath}
\usepackage{graphicx}
\usepackage{hyperref}
\usepackage{amsmath,arydshln,multirow}
\usepackage[alphabetic]{amsrefs}
\usepackage{arydshln}
\usepackage{cases}
\usepackage{amsmath}
\usepackage{amsfonts}
\usepackage{bm}
\usepackage{arydshln}
\usepackage{amsfonts,amsmath,amssymb,amscd,bbm,amsthm,mathrsfs,dsfont}
\usepackage{mathrsfs}
\usepackage{pb-diagram}
\usepackage{amssymb}
\usepackage[all,cmtip]{xy}
\usepackage{mathtools}
\usepackage{tikz}

\newtheorem{theorem}{Theorem}[section]

\newtheorem{proposition}[theorem]{Proposition}
\newtheorem{lemma}[theorem]{Lemma}

\theoremstyle{definition}
\newtheorem{definition}[theorem]{Definition}
\newtheorem{example}[theorem]{Example}
\newtheorem{proposition-definition}[theorem]{Proposition-Definition}
\newtheorem{definition-theorem}[theorem]{Definition-Theorem}

\newtheorem{corollary}[theorem]{Corollary}

\theoremstyle{remark}
\newtheorem{remark}[theorem]{Remark}

\numberwithin{equation}{section}

\def\mod{\opname{mod}\nolimits}

\newcommand{\opname}[1]{\operatorname{\mathsf{#1}}}

\newcommand{\Hom}{\opname{Hom}}

\newcommand{\End}{\opname{End}}

\newcommand{\Ext}{\opname{Ext}}

\newcommand{\Fac}{\opname{Fac}}

\newcommand{\add}{\opname{add}\nolimits}

\begin{document}

\title[Modules determined by their Newton polytopes]{Modules determined by their Newton polytopes}


\author{Peigen Cao}
\address{School of Mathematical Sciences, University of Science and Technology of China, Hefei, 230026, People's Republic of China
}
\email{peigencao@126.com}


\dedicatory{}


\date{}



\begin{abstract}
In the $\tau$-tilting theory, there exist two classes of foundamental modules: indecomposable $\tau$-rigid modules and left finite bricks. In this paper, we prove
the indecomposable $\tau$-rigid modules and the left finite bricks are uniquely determined by their Newton polytopes spanned by the dimensional vectors of their quotient modules. This is a kind of generalization of Gabriel's result that the indecomposable modules over path algebras of Dynkin quivers are uniquely determined by their dimensional vectors.
\end{abstract}

\maketitle


\tableofcontents

\section{Introduction}
 In 1972, Gabriel gave the foundamental result in the representation theory.
 \begin{theorem}[Gabriel's theorem]
 Let $K$ be an algebraically closed field and $A=KQ$ the path algebra of a Dynkin quiver $Q$. Then the dimensional vector $\underline{\dim}(N)$
of an indecomposable module $N$ in $\mod A$ gives a positive root in the root system $\Phi(Q)$ associated to $Q$ and 
 this correspondence gives a bijection from the indecomposable modules (up to
isomorphism) in $\mod A$ to the positive roots in $\Phi(Q)$.
 \end{theorem}

 Gabriel's theorem implies that the indecomposable modules over  $A=KQ$ are uniquely determined by their dimensional vectors.
 Since then, it has become a standard question in representation theory to identify which
class of modules are uniquely determined up to isomorphism by their composition series, or
equivalently by their dimension vectors, c.f., \cite{AR-1985,AD-2013,Mizuno-2014,Reid-2020}. 

In general, there is no satisfactory answer to the above question. Let us give an example.
\begin{example}\label{ex:A2}
  Consider the preprojective algebra $\Pi$ of type $A_2$, i.e., the algebra given by the following quiver
\[
\begin{tikzcd}
1 \arrow[r,bend left,"\alpha"]  & 2\arrow[l,bend left,swap,"\beta"]
\end{tikzcd}
\]
with relations $\alpha\beta=0,\;\beta\alpha=0$.
It is easy to see that 
\[0\rightarrow S_2\rightarrow P_1\rightarrow S_1\rightarrow 0\;\;\;\text{and}\;\;\;0\rightarrow S_1\rightarrow P_2\rightarrow S_2\rightarrow 0
\]
are exact sequences in $\mod \Pi$, where $S_i$ and $P_i$ denote the simple module and indecomposable projective module corresponding to vertex $i\in\{1,2\}$. Clearly, $\underline{\dim}(P_1)=(1,1)=\underline{\dim}(P_2)$ but $P_1$ and $P_2$ are not isomorphic. So the dimensional vector $(1,1)$ can not distinguish the two indecomposable projective modules $P_1$ and $P_2$.
\end{example}

\begin{definition}[Newton polytopes of modules]
    Let $A$ be a finite dimensional basic algebra over a field $K$ with $n$ simple modules. The Newton polytope $\mathcal N(U)$ of a module $U\in\mod A$ is  the convex hull in $\mathbb R^n$ spanned the dimensional vectors of ``quotient modules" of $U$. 
\end{definition}

\begin{remark}
   Recently, for totally different purpose, the combinatorics side of Newton polytopes of modules defined using submodules were studied by Fei in \cite{fei_2019a,fei_2019b}. For convenience of this paper, our Newton polytopes are defined  using quotient modules.
\end{remark}

Let us continue the Example \ref{ex:A2}. We know  the Newton polytope $\mathcal N(P_1)$ is the convex hull of $\{(1,1),\;(1,0),\;(0,0)\}$ and  the Newton polytope $\mathcal N(P_2)$ is the convex hull of $\{(1,1),\;(0,1),\;(0,0)\}$. Thus $\mathcal N(P_1)\neq \mathcal N(P_2)$. So it is natural to ask whether we can use the Newton polytopes to  distinguish some class of modules for general finite dimensional algebras. This is likely, because the Newton polytopes contain much more information than dimensional vectors. Indeed, it is easy to see that if two modules have the same Newton polytope, then they have the same dimensional vector.

\begin{definition}[Left finite modules]
 A module $N\in\mod A$ is said to be {\em left finite}, if the torsion class $\langle N\rangle_{\rm tors}$ generated by $N$ is functorially finite, equivalently, there exists a module $M\in\mod A$ such that $\langle N\rangle_{\rm tors}=\Fac M$.
\end{definition}

$\tau$-tilting theory was introduced by Adachi, Iyama and Reiten \cite{air_2014}, which completes the classic tilting theory from the viewpoint of mutations. Basic notions in $\tau$-tilting theory  will be recalled in Section \ref{sec:2}.
In the $\tau$-tilting theory, there exist two classes of foundamental modules: indecomposable $\tau$-rigid modules and left finite bricks.  Now let us state the main result in this paper.

\begin{theorem}[Theorem \ref{thm:main}]
Let $A$ be a finite dimensional basic algebra over a field $K$. Suppose that  $U$ and $V$ are two modules in $\mod A$ with the same Newton polytope. The following statements hold.

\begin{itemize}

    \item [(i)]  For any $\tau$-tilting pair $(M,P)$, $U\in\Fac M$ if and only if $V\in \Fac M$.
   \item [(ii)] If $U$ and $V$ are left finite, then $\langle U\rangle_{\rm tors}=\langle V\rangle_{\rm tors}$.
    \item[(iii)] If  $U$ and $V$ are  left finite bricks, then $U\cong V$.
    \item[(iv)] If $U$ and $V$ are indecomposable $\tau$-rigid modules, then $ U\cong V$.
     \item[(v)] If $U$ and $V$ are  $\tau$-rigid modules, then $U\oplus V$ is $\tau$-rigid.
\end{itemize}
\end{theorem}

An algebra $A$ is said to be {\em $\tau$-tilting finite}, if $\mod A$ has only finitely many torsion classes (See Definition \ref{def:tau-finite} and Theorem \ref{thm:DIJ-4.2} for details). Typical examples of $\tau$-tilting finite algebras contain the  preprojective algebras of Dynkin quivers, c.f., \cite{Mizuno-2014b}.  Since  all bricks are left finite for a $\tau$-tilting finite algebra, we have the following direct corollary.

 \begin{corollary}[Corollary \ref{cor:main}]
    Let $A$ be a $\tau$-tilting finite algebra. Then the bricks in $\mod A$ are uniquely determined by their Newton polytopes.
\end{corollary}

 {\bf Acknowledgement.} 
 This project have been partially supported  by grants from the National Key R\&D Program of China (2024YFA1013801), 
 the National Natural Science Foundation of China (Grant No. 12071422), and the Guangdong Basic and Applied Basic Research Foundation (Grant No. 2021A1515012035).

\section{$\tau$-tilting theory}\label{sec:2}

In this subsection, we recall $\tau$-tilting theory introduced by Adachi, Iyama and Reiten \cite{air_2014}. We fix a finite dimensional basic algebra $A$ over  a field $K$.  Denote by $\mod A$ the category of finitely generated right $A$-modules, and by $\tau$ the Auslander-Reiten translation in $\mod A$.  The  isomorphism classes of indecomposable projective modules in $\mod A$ are denoted by $P_1,\ldots,P_n$.

Given a module $M\in\mod A$, we denote by
\begin{itemize}
\item $\add M$  the additive closure of $M$ in $\mod A$;
\item $\Fac M$ the factor modules of the modules in $\add M$;
\item  $\prescript{\bot}{}{M}:=\{X\in \mod A  \mid \Hom_{A}(X,M)=0\}$ and $M^{\bot}:=\{Y\in \mod A  \mid \Hom_{A}(M,Y)=0\}$;
\item $|M|$ the number of non-isomorphic indecomposable direct summands of $M$, e.g., $|A|=n$.
\end{itemize}

\begin{definition}Let $M$ be a module in $\mod A$ and $P$ a projective module in $\mod A$.

\begin{itemize}
    \item [(i)] $M$ is called {\em $\tau$-rigid} if $\Hom_A(M,\tau M)=0$.
\item[(ii)] The pair $(M,P)$ is called \emph{$\tau$-rigid} if $M$ is $\tau$-rigid and $\Hom_A(P,M)=0$. 
\item[(iii)]  The pair $(M,P)$ is called \emph{$\tau$-tilting} if $(M,P)$ is a $\tau$-rigid pair and  $|M|+|P|=|A|$.
\end{itemize}
\end{definition}
 We will always consider modules and  $\tau$-rigid pairs up to isomorphism.  In a basic $\tau$-tilting pair $(M,P)$, it is known \cite[Proposition 2.3]{air_2014} that $P$ is uniquely determined by $M$.

A subcategory $\mathcal T$ of $\mod A$ is called a {\em torsion class}, if $\mathcal T$ is closed under quotients and extensions. A torsion class $\mathcal T$ is said to be {\em functorially finite}, if there exists a module $M\in\mod A$ such that $\mathcal T=\Fac M$.

Let $\mathcal C$ be a subcategory of $\mod A$. A module $U\in\mathcal C$ is said to be {\em Ext-projective} in $\mathcal C$, if $\Ext_A^1(U,\mathcal C)=0$. We denote by $\mathcal P(\mathcal C)$ the direct sum of one copy of each of the indecomposable Ext-projective objects in $\mathcal C$ up to isomorphisms.

\begin{theorem}[\cite{air_2014}*{Proposition 1.2 (b) and Theorem 2.7}] \label{thm:air-torsion}
The following statements hold.
\begin{itemize}
    \item [(i)] There is a well-defined map $\Psi$ 
from $\tau$-rigid pairs to functorially finite torsion classes in $\mod A$
given by $(M,P)\mapsto \Fac M$.

\item[(ii)] The above map $\Psi$ is a bijection if we restrict it to basic $\tau$-tilting pairs.

\item[(iii)] Let $\mathcal T$ be a functorially finite torsion class and denote by $(M,P)$ the basic $\tau$-tilting pair such that $\Fac M=\mathcal T$. Then  $M=\mathcal P(\mathcal T)$.
\end{itemize}
\end{theorem}

\begin{proposition}[\cite{air_2014}*{{Proposition 2.9 and Theorem 2.10}}]\label{pro:minmax}
 Let $(U,Q)$ be a basic $\tau$-rigid pair and  $(M,P)$ a basic $\tau$-tilting pair in $\mod A$. Then
 \begin{itemize}
 \item[(i)]  $\Fac U$ and  $\prescript{\bot}{}({\tau U})\cap Q^\bot$ are functorially finite torsion classes in $\mod A$.
 \item[(ii)]  $\Fac U\subseteq \Fac M\subseteq \prescript{\bot}{}({\tau U})\cap Q^\bot$ if and only if $(U,Q)$ is a direct summand of the basic $\tau$-tilting pair $(M,P)$.
 \end{itemize}
\end{proposition}

Consider $M\in\mod A$ and let 
\[\bigoplus_{i=1}^nP_i^{b_i}\rightarrow
\bigoplus_{i=1}^nP_i^{a_i}\rightarrow M\rightarrow 0
\]
be the minimal projective presentation of $M$ in $\mod A$. The
vector $$\delta_M:=(a_1-b_1,\ldots,a_n-b_n)^T\in\mathbb Z^n$$ is called the {\em $\delta$-vector} of $M$ and the vector ${\bf g}_M:=-\delta_M$ is called the {\em $g$-vector} of $M$.

For a $\tau$-rigid pair $(M,P)$, we define its  $\delta$-vector and $g$-vector as follows:
\[ 
    \delta_{(M,P)}:=\delta_M-\delta_P,\;\; {\bf g}_{(M,P)}:=-\delta_{(M,P)}.
\]
With this definition, one can see that the $g$-vector ${\bf g}_{(0,P_k)}$ of $(0,P_k)$ is the $k$th column of $I_n$.

\begin{remark}
    Notice that the $\delta$-vectors defined here coincide with the $g$-vectors used in \cite{air_2014} and they correspond to the negative of $g$-vectors in cluster algebras under categorifications. 
\end{remark}

For a module $N\in\mod A$, denote by $\langle N\rangle_{\rm tors}$ the torsion classes of $\mod A$ generating by $N$, that is, $\langle N\rangle_{\rm tors}$ is the smallest torsion class containing $N$.

Recall that a module $N\in\mod A$ is said to be {\em left finite}, if the torsion class $\langle N\rangle_{\rm tors}$ is functorially finite. For example, $\tau$-rigid modules are left finite.
 A module $S\in\mod A$ is said to be a {\em brick}, if $\End_A(S)$ is a division algebra. For example, simple modules are bricks.

\begin{theorem}[Brick-$\tau$-rigid correspondence, \cite{DIJ-17}*{Section 4}] \label{thm:dij-corr}
Let $N$ be an indecomposable $\tau$-rigid module in $\mod A$. Then the following statements hold.
    \begin{itemize}
        \item [(i)] There is a unique brick $S\in\mod A$ satisfying $\Fac N=\langle S\rangle_{\rm tors}$. In particular, the brick $S$ is left finite;
        \item[(ii)] The correspondence $N\mapsto S$ gives a bijection from the indecomposable $\tau$-rigid modules in $\mod A$ to the left finite bricks in $\mod A$.
    \end{itemize}
\end{theorem}

\begin{definition}[\cite{DIJ-17}*{Definition 1.1}]\label{def:tau-finite}
We say that $A$ is {\em $\tau$-tilting finite} if there are only finitely many basic $\tau$-tilting pairs in $\mod A$ up to isomorphism, equivalently, $\mod A$ has only finitely many functorially finite torsion classes.
\end{definition}

\begin{theorem}[\cite{DIJ-17}*{Theorem 4.2}]
\label{thm:DIJ-4.2}
Let $A$ be a finite dimensional algebra. The following statements are equivalent.
\begin{itemize}
    \item [(a)] The algebra $A$ is $\tau$-tilting finite;
    \item[(b)] The set of bricks in $\mod A$ are finite;
    \item[(c)] The set of left finit bricks in $\mod A$ are finite.
\end{itemize}
Moreover, if $A$ is $\tau$-tilting finite, then  all torsion classes of $\mod A$ are functorially finite and all bricks in $\mod A$ are left finite.
\end{theorem}

\section{Modules determined by their Newton polytopes}

For a module $X\in \mod A$ and a vector $\delta\in\mathbb R^n$, we write $\delta(X)\in\mathbb R$ for the inner product of $\delta$ and the dimensional vector of $X$.

\begin{lemma}\label{lem:delta}
 Denote by $$\overline{\mathcal T}_\delta=\{ N\in\mod A\mid \text{for any quotient module X of }N,\;\; \delta(X)\geq 0
     \}$$ for $\delta\in\mathbb R^n$.
 The following statements hold.
     \begin{itemize}
         \item [(i)] \cite{BKT-2014}*{Proposition 3.1} The subcategory $\overline{\mathcal T}_\delta$ is a torsion class in $\mod A$.
         \item[(ii)] \cite[Proposition 3.11]{asai-2021} Let $(M,P)$ be a basic $\tau$-tilting pair in $\mod A$ and $\delta=\delta_{(M,P)}$  the $\delta$-vector of $(M,P)$. Then $\Fac M=\overline{\mathcal T}_\delta$. 
     \end{itemize}
\end{lemma}

Torsion classes of the form $\overline{\mathcal T}_\delta$ are called {\em semistable torsion classes}. By Lemma \ref{lem:delta} (ii),  all functorially finite torsion classes are semistable torsion classes.

Recall that the {\em Newton polytope} $\mathcal N(U)$ of a module $U\in\mod A$ is  the convex hull in $\mathbb R^n$ spanned the dimensional vectors of quotient modules of $U$.

\begin{theorem}\label{thm:main}
Let $A$ be a finite dimensional basic algebra over a field $K$. Suppose that  $U$ and $V$ are two modules in $\mod A$ with the same Newton polytope. The following statements hold.

\begin{itemize}

    \item [(i)] For any semistable torsion class $\overline{\mathcal T}_\delta$ , $U\in\overline{\mathcal T}_\delta$ if and only if $V\in\overline{\mathcal T}_\delta$. In particular, for any $\tau$-tilting pair $(M,P)$, $U\in\Fac M$ if and only if $V\in \Fac M$.
   \item [(ii)] If $U$ and $V$ are left finite, then $\langle U\rangle_{\rm tors}=\langle V\rangle_{\rm tors}$.
    \item[(iii)] If  $U$ and $V$ are  left finite bricks, then $U\cong V$.
    \item[(iv)] If $U$ and $V$ are indecomposable $\tau$-rigid modules, then $ U\cong V$.
     \item[(v)] If $U$ and $V$ are  $\tau$-rigid modules, then $U\oplus V$ is $\tau$-rigid.
\end{itemize}
\end{theorem}
\begin{proof}
(i) Suppose $U\in\overline{\mathcal T}_\delta$. Then for any quotient module $X$ of $U$, we have $\delta(X)\geq 0$. This implies for any vector $\alpha$ in the Newton polytope of $U$, the inner product $\langle\delta,\alpha \rangle$ is non-negative. Since $U$ and $V$ have the same Newton polytope, we know that for any quotient module $Y$ of $V$, we have $\delta(Y)\geq 0$. This implies $V\in \overline{\mathcal T}_\delta$.

Conversely, suppose $V\in \overline{\mathcal T}_\delta$ and by the same arguments, we can show $U\in\overline{\mathcal T}_\delta$. Hence,  $U\in\overline{\mathcal T}_\delta$ if and only if $V\in\overline{\mathcal T}_\delta$. Since all functorially finite torsion classes are semistable torsion classes, we obtain that $U\in\Fac M$ if and only if $V\in \Fac M$.

(ii) Since $U$ is left finite, the torsion class $\langle U\rangle_{\rm tors}$ is functorially finite.  Then by Theorem \ref{thm:air-torsion} (ii), there exists a $\tau$-tilting pair $(M,P)$ such that $\Fac M=\langle U\rangle_{\rm tors}$.
 Then by (i) and the fact $U\in \langle U\rangle_{\rm tors}=\Fac M$, we have $V\in \Fac M=\langle U\rangle_{\rm tors}$. Thus  $\langle V\rangle_{\rm tors}\subseteq \langle U\rangle_{\rm tors}$.

Applying the same arguments to $V$, we can show the converse inclusion  $\langle U\rangle_{\rm tors}\subseteq \langle V\rangle_{\rm tors}$. Thus $\langle U\rangle_{\rm tors}= \langle V\rangle_{\rm tors}$.

(iii) By (ii), we have $\langle U\rangle_{\rm tors}=\langle V\rangle_{\rm tors}$. Since $U$ is a left finite brick and by the brick-$\tau$-rigid correspondence in Theorem \ref{thm:dij-corr}, there exists an indecomposable $\tau$-rigid module $N$ such that $$\Fac N=\langle U\rangle_{\rm tors}=\langle V\rangle_{\rm tors}.$$ Then by the uniqueness of the brick in Theorem \ref{thm:dij-corr} (i), we have $U\cong V$.

(iv) Since $U$ and $V$ are indecomposable $\tau$-rigid modules, we know that they are left finite. By (ii), we have $$\Fac U=\langle U\rangle_{\rm tors}=\langle V\rangle_{\rm tors}=\Fac V.$$
Then by Theorem \ref{thm:dij-corr} (i), there exists a unique brick $S\in\mod A$ such that $\Fac U=\langle S\rangle_{\rm tors}=\Fac V$.
Then by the brick-$\tau$-rigid correspondence in Theorem \ref{thm:dij-corr}, we obtain $U\cong V$.

(v) Since $U$ and $V$ are $\tau$-rigid modules, we know that they are left finite. By (ii), we have $$\Fac U=\langle U\rangle_{\rm tors}=\langle V\rangle_{\rm tors}=\Fac V.$$ 
By Theorem \ref{thm:air-torsion} (ii), there exists a basic $\tau$-tilting pair $(M,P)$ such that 
$\Fac M=\Fac U=\Fac V$. Since $\Fac M=\Fac U\subseteq \prescript{\bot}{}({\tau U})$ and 
$\Fac M=\Fac V\subseteq \prescript{\bot}{}({\tau V})$ and 
by Proposition \ref{pro:minmax} (ii), we have $U\in\add M$ and $V\in\add M$. Thus $U\oplus V$ is $\tau$-rigid.
\end{proof}

\begin{corollary}\label{cor:main}
  Let $A$ be a $\tau$-tilting finite algebra. Then the bricks in $\mod A$ are uniquely determined by their Newton polytopes.
\end{corollary}
\begin{proof}
    Since $A$ is $\tau$-tilting finite and by Theorem \ref{thm:DIJ-4.2}, all bricks in $\mod A$ are left finite. Then the result follows from Theorem \ref{thm:main} (iii).
\end{proof}

\bibliographystyle{alpha}
\bibliography{myref}

\end{document}